\documentclass[reqno]{amsart}

\usepackage{enumerate}
\usepackage{tabto}
\usepackage[mathscr]{euscript}
\usepackage{xcolor}
\usepackage{layout}
\usepackage{fancyhdr}
\usepackage{array}
\usepackage{amsfonts}
\usepackage{amsmath}
\usepackage{amssymb}
\usepackage{mathtools}
\usepackage{graphicx}
\usepackage{bm}
\usepackage{enumitem}
\usepackage{caption} 
\usepackage{color}
\usepackage{kotex}
\usepackage{csquotes}
\usepackage{bookmark}
\usepackage{float}
\usepackage{multirow}
\usepackage[square,numbers,sort&compress]{natbib}
\usepackage{hyperref}
\hypersetup{colorlinks=true,linkcolor=blue,citecolor=red}
\allowdisplaybreaks

\def\XXint#1#2#3{{\setbox0=\hbox{$#1{#2#3}{\int}$ }
\vcenter{\hbox{$#2#3$ }}\kern-.6\wd0}}

\newtheorem{theorem}{Theorem}[section]
\newtheorem{lemma}[theorem]{Lemma}

\newtheorem{proposition}[theorem]{Proposition}
\theoremstyle{definition}
\newtheorem{definition}[theorem]{Definition}
\newtheorem{remark}[theorem]{Remark}

\numberwithin{equation}{section}

\newcommand{ \mc }{ \mathcal }
\newcommand{ \mr }{ \mathbb{R} }

\newcommand{ \dist }{ \operatorname{dist} }
\newcommand{ \loc }{ \operatorname{loc} }

\newcommand{\iints}[1]{{\int\hspace{-0.28cm}\int_{#1}}}

\begin{document}
\title[absence of the Lavrentiev phenomenon]{absence of the Lavrentiev phenomenon for a general class of parabolic double phase problems}

\author{Youngchae Kim}\address{Department of Mathematics, Kyungpook National University, Daegu, 41566, Republic of Korea} \email{rjseka5241@knu.ac.kr} \author{Jehan Oh}\address{Department of Mathematics, Kyungpook National University, Daegu, 41566, Republic of Korea} \email{jehan.oh@knu.ac.kr}

\date{\today.}
\thanks{Youngchae Kim is supported by the National Research Foundation of Korea (NRF) grant funded by the Korea government [Grant No. RS-2025-25433309].
Jehan Oh is supported by the National Research Foundation of Korea (NRF) grant funded by the Korea government [Grant Nos. RS-2025-00555316 and RS-2025-25415411].}

\begin{abstract}
In this paper, we prove the absence of the Lavrentiev phenomenon for a general class of parabolic double phase functionals with Orlicz growth. The energy density is given by
$$
G(|Dw|)+a(x,t)H(|Dw|),
$$
where $G$ and $H$ are Young functions satisfying the $\Delta_2$ and $\nabla_2$ conditions with $G\prec H$, and $a(\cdot)$ is a continuous nonnegative coefficient. Under suitable balance conditions between the growth gap of $G$ and $H$ and the modulus of continuity of $a(\cdot)$, we show that every finite-energy map can be approximated locally by smooth functions without loss of energy. The result extends the known parabolic double phase theory from power type growth to a broad Young function framework and identifies the natural space-time Orlicz energy class for the problem.
\end{abstract}
\maketitle

\section{Introduction}

In this paper, we investigate the absence of the Lavrentiev phenomenon for a general class of parabolic double phase functionals given by
\begin{equation}    \label{eq: main functional}
	L^1(0,T;W^{1,1}(\Omega))\ni w\to \mc{P}(w,\Omega_T)\coloneq\iints{\Omega_T}[G(|Dw|)+a(z)H(|Dw|)]\,dz,
\end{equation}
where $\Omega_T \coloneq \Omega\times(0,T)$ is a parabolic cylinder with a bounded open set $\Omega\subset \mr^n$ and $T>0$. Here, $G,H:[0,\infty)\to[0,\infty)$ are Young functions with $G,H\in\Delta_2\cap\nabla_2$ and $G\prec H$ (see Section \ref{subsec: Young functions}) and $a(\cdot):\Omega_T\to[0,\infty)$ is a continuous function with a modulus of continuity $\omega$. The Lavrentiev phenomenon is one of the important topics in the calculus of variations. For a functional $\mc{G}$ defined on a function space $X$ and for a dense subspace $Y$ of $X$, if
$$
\inf_X\mc{G}<\inf_Y\mc{G},
$$
then we say that the Lavrentiev phenomenon occurs. For instance, when $G:\Omega\times\mr^n\to\mr$ is a Carath\'{e}odory function satisfying non-standard growth conditions (in Marcellini's terminology \cite{Marcellini1989,Marcellini2020,Marcellini1991}) of $(p,q)$-type
$$
|\xi|^p\lesssim G(x,\xi)\lesssim |\xi|^q+1,\quad 1<p<q,
$$
for non-autonomous functionals of the type
$$
u\mapsto\mc{G}_0(u,\Omega)\coloneq\int_\Omega G(x,Du)\,dx,
$$
the Lavrentiev phenomenon occurs in a ball $B\subset\subset\Omega$:
$$
\inf_{u\in u_0+W_0^{1,p}(B)}\mc{G}_0(u,B)<\inf_{u\in u_0+W_0^{1,p}(B)\cap W_{\loc}^{1,q}(B)}\mc{G}_0(u,B),
$$
where $u_0$ is a boundary data with proper regularity. According to \cite{Mingione2021}, it not only represents an obstruction to regularity, but also indicates that the same assumptions ensuring a priori regularity estimates for minimizers can be used to prove the absence of the Lavrentiev phenomenon. This relationship can be reversed, so that the absence of the Lavrentiev phenomenon can be exploited to derive regularity estimates. In this sense, we say that the Lavrentiev phenomenon does not occur if for any $w\in W^{1,p}(B)$, there exists a sequence of functions $\left\{ u_k \right\}$ in a smoother space than $W^{1,p}(B)$ such that
$$
\mc{G}_0(u_k,B)\to\mc{G}_0(u,B)
$$
for any ball $B\subset\subset\Omega$.

\medskip The elliptic double phase equation is given by
$$
-\operatorname{div}(|Du|^{p-2}Du+a(x)|Du|^{q-2}Du)=0\quad\text{in }\Omega
$$
where $\Omega\subset\mr^n$ is open, $1<p<q$ and $0\leq a(\cdot)\in C^\alpha(\Omega)$ with $\alpha\in(0,1]$. This is motivated by strongly anisotropic materials. It was first introduced by Zhikov in \cite{Zhikov1986,Zhikov1993,Zhikov1995,Zhikov1997} and the corresponding functional
$$
W^{1,1}(\Omega)\ni w\mapsto\int_\Omega[|Dw|^p+a(x)|Dw|^q]\,dx
$$
is an example of the Lavrentiev phenomenon, see \cite{Esposito2004,Fonseca2004,Zhikov1995,Zhikov1997}. The Lavrentiev phenomenon for elliptic double phase problems has been actively studied. For instance, \cite{Esposito2004} and \cite{Colombo2015} show that the gap condition
$$
\frac{q}{p}\leq 1+\frac{\alpha}{n}
$$
is equivalent to the absence of the Lavrentiev phenomenon. Moreover, the assumptions
$$
u\in L^\infty(\Omega)\quad\text{and}\quad q\leq p+\alpha
$$
and
$$
u\in W^{1,q}(\Omega)\quad\text{and}\quad q\leq p+\alpha\max\left\{ 1,\frac{p}{n} \right\}
$$
also ensure absence of the Lavrentiev phenomenon, see \cite{Colombo2015a,Esposito2004,Fonseca2004,Bulicek2022}. In addition, regularity theory as well as the Lavrentiev phenomenon for elliptic double phase problems has been extensively studied. Furthermore, elliptic double phase problems with two modulating coefficients and generalized double phase problems have also been widely investigated. For instance, local H\"{o}lder continuity for weak solutions and their gradients are derived in \cite{Baroni2015,Baroni2018,Ok2017}, while Calder\'{o}n-Zygmund estimates are established in \cite{Byun2017,Baasandorj2020,Colombo2016,DeFilippis2019,Byun2021a,Kim2026}. For further results on regularity of elliptic double phase problems, see \cite{Ok2020,Byun2020,Byun2021,DeFilippis2023,Haestoe2022,Haestoe2022a,Kim2024a,Kim2025a}.

On the other hand, the parabolic double phase equation is given by
$$
u_t-\operatorname{div}(|Du|^{p-2}Du+a(x,t)|Du|^{q-2}Du)=0\quad\text{in }\Omega_T\coloneq \Omega\times(0,T),
$$
where $n\geq 2$, $\frac{2n}{n+2}<p<q<\infty$, $\Omega$ is a bounded open subset of $\mr^n$, and $T>0$. Here, $a(\cdot):\Omega_T\to[0,\infty)$ is a function in $C^{\alpha,\frac{\alpha}{2}}(\Omega_T)$ which means that $a(\cdot)\in  L^\infty(\Omega_T)$ and there is some constant $[a]_{\alpha,\frac{\alpha}{2}}$ such that
$$
|a(x,t)-a(y,s)|\leq[a]_{\alpha,\frac{\alpha}{2}}\max\left\{ |x-y|^\alpha,|t-s|^\frac{\alpha}{2} \right\}
$$
for any $(x,t),(y,s)\in\Omega_T$. Research on parabolic double phase problems has only recently become active. In particular, while some progress has been made on regularity issues such as higher integrability (see, e.g., \cite{Wontae2023,Wontae2024,Chlebicka2025a,Kim2025b,Kim2026a}), results concerning the Lavrentiev phenomenon remain limited, with essentially only \cite{Chlebicka2019a,Chlebicka2021,BYJ2026} available. In \cite{BYJ2026}, it is proved that for $w\in L^p(0,T;W^{1,p}(\Omega))$, if one of the conditions
$$
\begin{aligned}
	& (1)\quad q\leq p+\frac{p\alpha}{n+2}, \\
	& (2)\quad w\in L^\infty(\Omega_T) \quad\text{and}\quad q\leq p+\max\left\{ \frac{p\alpha}{n+2},\alpha \right\}, \\
	& (3)\quad w\in C(0,T;L^s(\Omega)),\quad s\geq 2,\quad\text{and}\quad q\leq p+\max\left\{ \frac{s\alpha}{n+s},\frac{p\alpha}{n+2} \right\}
\end{aligned}
$$
holds, then the Lavrentiev phenomenon does not occur. Also, the absence of the Lavrentiev phenomenon for parabolic equations with anisotropic Musielak-Orlicz spaces is established in \cite{Chlebicka2019a} and \cite{Chlebicka2021}.

\medskip In the study of problems with non-standard growth, such as double phase problems, the choice of the underlying function space is determined not only by formal analogies with the power-type case, but more importantly by the structure of the energy functional and the approximation properties required to investigate the Lavrentiev phenomenon.

At first glance, a natural extension of the classical assumption
$$
w\in L^p(0,T;W^{1,p}(\Omega))
$$
would be to consider
$$
w\in L^G(0,T;W^{1,G}(\Omega)).
$$
However, this choice turns out to be inappropriate in the non-standard growth setting. The reason is that $L^G(0,T;W^{1,G}(\Omega))$ is a mixed norm space, where the integrability in time is evaluated after measuring the spatial norm. This does not correspond to the actual energy functional
$$
\iints{\Omega_T}[G(|Dw|)+a(z)H(|Dw|)]\,dz.
$$
In general Orlicz spaces, such mixed norm spaces do not coincide with the corresponding space-time spaces; that is,
$$
L^G(0,T;L^G(\Omega))\neq L^G(\Omega_T),
$$
unless $G$ is equivalent to some power function $t^p$ (see \cite{Maligranda2004}). This remains true even if $G\in\Delta_2\cap\nabla_2$. Consequently, this mismatch leads to a functional framework incompatible with the variational structure of the problem.

Moreover, the analysis of the Lavrentiev phenomenon is fundamentally connected with the density and approximation properties of the admissible space. In Orlicz spaces, the density of smooth functions may fail, and this failure is precisely what gives rise to the Lavrentiev phenomenon. Therefore, the function space must be chosen so that it reflects the actual domain of the energy functional, while still allowing one to study approximation by regular functions (see \cite{Chlebicka2019a,Chlebicka2019b}).

For this reason, the natural setting is given by
$$
w\in L^1(0,T;W^{1,1}(\Omega))
$$
together with the finiteness of the energy
$$
\mc{P}(w,\Omega_T)=\iints{\Omega_T}[G(|Dw|)+a(z)H(|Dw|)]\,dz<\infty.
$$
Note that $|Dw|\in L^G(\Omega_T)$ follows from the finiteness of the functional $\mc{P}$ and that this space can be regarded as a natural minimal framework.

\begin{theorem}    \label{thm: main theorem}
    Let $w\in L^1(0,T;W^{1,1}(\Omega))$ be such that $\mathcal{P}(w,\Omega_T)<\infty$ and let $Q\subset\subset\Omega_T$ be any cylinder. Assume that one of the following conditions holds:
    \begin{enumerate}[label=\upshape(\roman*)]
        \item\label{thm: main theorem general case} the gap condition
        \begin{equation}    \label{eq: general gap condition}
			\limsup_{\rho\to 0^+}\omega(\rho)\frac{(H\circ G^{-1})(\rho^{-n-2})}{\rho^{-n-2}}<\infty;
		\end{equation}
        \item\label{thm: main theorem bounded case} $w\in L^\infty(\Omega_T)$ and
        \begin{equation}    \label{eq: bounded gap condition}
			\limsup_{\rho\to 0^+}\omega(\rho)\frac{H(\rho^{-1})}{G(\rho^{-1})}<\infty;
		\end{equation}
        \item\label{thm: main theorem L^s case} $w\in L^\infty(0,T;L^s(\Omega))$ with $s\geq 2$ and
		\begin{equation}    \label{eq: L^infty(L^s) gap condition}
			\limsup_{\rho\to 0^+}\omega(\rho)\frac{H(\rho^{-(n+s)/s})}{G(\rho^{-(n+s)/s})}<\infty.
		\end{equation}
    \end{enumerate}
    Then there exists a sequence $\left\{ w_m \right\}_{m=1}^\infty$ in $C^\infty(Q)$ such that
    $$
    w_m\rightarrow w\text{ in }L^1(Q),\quad Dw_m\to Dw\text{ in }L^G(Q;\mr^n),\quad \mathcal{P}(w_m,Q)\rightarrow\mathcal{P}(w,Q).
    $$
\end{theorem}

\begin{remark}
	Indeed, if $G(t)=t^p$, $H(t)=t^q$ and $a(\cdot)\in C^{\alpha,\frac{\alpha}{2}}(\Omega_T)$ with $2\leq p<q$ and $\alpha\in(0,1]$, then it is easy to check that
	$$
	\text{the condition }\eqref{eq: general gap condition} \iff q\leq p+\frac{p\alpha}{n+2},
	$$
	$$
	\text{the condition }\eqref{eq: bounded gap condition} \iff q\leq p+\alpha,
	$$
	and
	$$
	\text{the condition }\eqref{eq: L^infty(L^s) gap condition} \iff q\leq p+\frac{s\alpha}{n+s}.
	$$
\end{remark}

\medskip The outline of this paper is as follows. Section \ref{sec: preliminaries} introduces the definitions and properties of Young functions and the parabolic mollification used throughout this paper, while Section \ref{sec: proof of the main theorem} provides the proof of Theorem \ref{thm: main theorem}.

\section{Preliminaries}\label{sec: preliminaries}

\subsection{Young functions}\label{subsec: Young functions}

A Young function $\Phi:[0,\infty)\to[0,\infty)$ is an increasing convex function satisfying
$$
\Phi(0)=0,\quad\lim_{t\to \infty}\Phi(t)=\infty,\quad\lim_{t\to 0^+}\frac{\Phi(t)}{t}=0,\quad\lim_{t\to\infty}\frac{\Phi(t)}{t}=\infty.
$$

\begin{definition}    \label{def: Delta_2 and nabla_2 conditions}
	Let $\Phi$ be a Young function.
	\begin{enumerate}[label=\upshape(\arabic*)]
		\item\label{def: Delta_2} $\Phi$ is said to satisfy the $\Delta_2$-condition, $\Phi\in\Delta_2$, if there exists a constant $\Delta_2(\Phi)>0$ such that $\Phi(2t)\leq\Delta_2(\Phi)\Phi(t)$ for all $t\geq 0$.
		\item\label{def: nabla_2} $\Phi$ is said to satisfy the $\nabla_2$-condition, $\Phi\in\nabla_2$, if there exists a constant $\nabla_2(\Phi)>1$ such that $\Phi(\nabla_2(\Phi)t)\geq 2\nabla_2(\Phi)\Phi(t)$ for all $t\geq 0$.
		\item\label{def: Delta_2 and nabla_2} We write $\Phi\in\Delta_2\cap\nabla_2$ if $\Phi\in\Delta_2$ and $\Phi\in\nabla_2$.
	\end{enumerate}
\end{definition}

Note that if $\Phi\in\Delta_2$, then $\Delta_2(\Phi)>2$. More precisely, by the convexity of $\Phi$, we have
$$
\Phi(2t)\leq\Delta_2(\Phi)\Phi(t)\leq\frac{\Delta_2(\Phi)}{2}\Phi(2t)
$$
for all $t\geq 0$; hence, $\Delta_2(\Phi)\geq 2$. If $\Delta_2(\Phi)=2$, it follows from the previous inequality that $\Phi(2t)=2\Phi(t)$ for all $t\geq 0$ and thus
$$
	\Phi(1)=\frac{\Phi(2^k)}{2^k}
$$
for any $k\in\mathbb{Z}$. Since $\Phi$ is convex and $\Phi(0)=0$, the map $t\mapsto\Phi(t)/t$ is increasing and this implies that $\Phi(t)/t=\Phi(1)$ for $t> 0$; that is, $\Phi(t)=\Phi(1)t$ for $t\geq 0$ which is not a Young function. Therefore, $\Delta_2(\Phi)>2$.

\medskip The definitions of $\Delta_2$ and $\nabla_2$ conditions directly imply the following lemma.

\begin{lemma}[\cite{Byun2020}, Lemma 2.2]    \label{lem: properties for Young functions satisfying Delta_2 and nabla_2 conditions}
	Let $\Phi$ be a Young function.
	\begin{enumerate}[label=\upshape(\arabic*)]
		\item\label{lem: Delta_2 means aDec} For any $1\leq\Lambda<\infty$ and $t\geq 0$, we have
		$$
		\Phi(\Lambda t)\leq\Delta_2(\Phi)^{1+\log_2 \Lambda}\Phi(t),
		$$
		provided with $\Phi\in\Delta_2$.
		\item\label{lem: nabla_2 means aInc} For any $0<\lambda\leq 1$ and $t\geq 0$, we have
		$$
		\Phi(\lambda t)\leq 2\nabla_2(\Phi)\lambda^{1+\log_{\nabla_2(\Phi)}2}\Phi(t),
		$$
		provided with $\Phi\in\nabla_2$.
	\end{enumerate}
\end{lemma}

We now introduce a partial order relation between Young functions and some lemmas which will be used in Section \ref{sec: proof of the main theorem}.

\begin{definition}    \label{def: partially order for Young functions}
	Let $\Phi_1$ and $\Phi_2$ be Young functions. We shall write
	$$
	\Phi_1\prec\Phi_2
	$$
	if $\Phi_2\circ\Phi_1^{-1}$ is a Young function, where $\Phi_1^{-1}$ is the left-inverse of $\Phi_1$; that is,
    $$
    \Phi_1^{-1}(t)=\inf\left\{ s \geq 0 \;:\; \Phi_1(s) \geq t \right\}.
    $$
\end{definition}

The relation means that $\Phi_2$ grows faster than $\Phi_1$.

\begin{lemma}[\cite{Byun2020}, Lemma 2.7]    \label{lem: fraction of two Young functions}
	Let $\Phi_1$ and $\Phi_2$ be Young functions with $\Phi_1\prec\Phi_2$. Then
	$$
	t\mapsto\left( \frac{\Phi_2}{\Phi_1} \right)(t)=\frac{\Phi_2(t)}{\Phi_1(t)}
	$$
	is non-decreasing.
\end{lemma}

\begin{lemma}    \label{lem: properties for Phi_1,Phi_2 satisfy Delta_2, nabla_2, Phi_1 < Phi_2}
	Let $\Phi_1$ and $\Phi_2$ be Young functions with $\Phi_1,\Phi_2\in\Delta_2\cap\nabla_2$ and $\Phi_1\prec\Phi_2$. Then
	\begin{enumerate}[label=\upshape(\alph*)]
		\item\label{lem: Phi_2 circ Phi_1^{-1} satisfies Delta_2} $\Phi_2\circ \Phi_1^{-1}\in\Delta_2$;
		\item\label{lem: lem: almost decreasing-type inequality for Phi_2 circ Phi_1^{-1} and Phi_2/Phi_1} if $\Lambda\geq 1$ is given, there exists a constant $C\equiv C(\nabla_2(\Phi_1),\Delta_2(\Phi_2),\Lambda)>0$ such that
		$$
		(\Phi_2\circ \Phi_1^{-1})(\Lambda t)\leq C(\Phi_2\circ \Phi_1^{-1})(t)\quad\text{and}\quad\left( \frac{\Phi_2}{\Phi_1} \right)(\Lambda t)\leq C\left( \frac{\Phi_2}{\Phi_1} \right)(t)
		$$
		for all $t>0$.
	\end{enumerate}
\end{lemma}
\begin{proof}
	Since $\Phi_1\in\nabla_2$, for any $t>0$, we have
	$$
	\Phi_1(\nabla_2(\Phi_1) \Phi_1^{-1}(t))\geq 2\nabla_2(\Phi_1)\Phi_1(\Phi_1^{-1}(t))=2\nabla_2(\Phi_1)t> 2t.
	$$
	Since $\Phi_1^{-1}$ is increasing, it follows that
	$$
	\nabla_2(\Phi_1)\Phi_1^{-1}(t)=\Phi_1^{-1}(\Phi_1(\nabla_2(\Phi_1) \Phi_1^{-1}(t)))\geq\Phi_1^{-1}(2t).
	$$
	Applying Lemma \ref{lem: properties for Young functions satisfying Delta_2 and nabla_2 conditions}\ref{lem: Delta_2 means aDec} to $\Phi_2$, we obtain
	$$
	\Phi_2(\Phi_1^{-1}(2t))\leq \Phi_2(\nabla_2(\Phi_1)\Phi_1^{-1}(t))\leq\Delta_2(\Phi_2)^{1+\log_2\nabla_2(\Phi_1)}\Phi_2(\Phi_1^{-1}(t));
	$$
	this implies $\Phi_2\circ\Phi_1^{-1}\in\Delta_2$ with $\Delta_2(\Phi_2\circ\Phi_1^{-1})=\Delta_2(\Phi_2)^{1+\log_2\nabla_2(\Phi_1)}$.

	Now, we fix $\Lambda \geq 1$ and set $\Phi \coloneq \Phi_2\circ\Phi_1^{-1}\in\Delta_2$. Using Lemma \ref{lem: properties for Young functions satisfying Delta_2 and nabla_2 conditions}\ref{lem: Delta_2 means aDec} once more, we have
	$$
	\Phi(\Lambda t)\leq\Delta_2(\Phi)\Lambda^{\log_2\Delta_2(\Phi)}\Phi(t)
	$$
	and hence
	$$
	\frac{\Phi(\Lambda t)}{\Lambda t}\leq\Delta_2(\Phi)\Lambda^{\log_2\Delta_2(\Phi)-1}\frac{\Phi(t)}{t}
	$$
	with $\Delta_2(\Phi)>2$. On the other hand, it follows from Lemma \ref{lem: properties for Young functions satisfying Delta_2 and nabla_2 conditions} that
	$$
	\Phi_1(t)\leq \frac{1}{\Lambda_0}\Phi_1(\Lambda t)\quad\text{and}\quad \Phi_2(\Lambda t)\leq\Delta_2(\Phi_2)^{1+\log_2\Lambda}\Phi_2(t),
	$$
	where
	$$
	\frac{1}{\Lambda_0} \coloneq 2\nabla_2(\Phi_1)\left( \frac{1}{\Lambda} \right)^{1+\log_{\nabla_2(\Phi_1)}2}>0.
	$$
	Therefore,
	$$
	\left( \frac{\Phi_2}{\Phi_1} \right)(\Lambda t)\leq\Delta_2(\Phi_2)^{1+\log_2\Lambda}\frac{1}{\Lambda_0}\left( \frac{\Phi_2}{\Phi_1} \right)(t).
	$$
\end{proof}

\begin{remark}
	In the proof of Lemma \ref{lem: properties for Phi_1,Phi_2 satisfy Delta_2, nabla_2, Phi_1 < Phi_2}\ref{lem: Phi_2 circ Phi_1^{-1} satisfies Delta_2}, we only used the assumptions $\Phi_1\in\nabla_2$ and $\Phi_2\in\Delta_2$, not $\Phi_1,\Phi_2\in\Delta_2\cap\nabla_2$. Moreover, it does not follow in general that $\Phi_2\circ\Phi_1^{-1}\in\nabla_2$, even if $\Phi_1,\Phi_2\in\Delta_2\cap\nabla_2$. For instance, let
	$$
	\Phi_1(t)=t^p\quad\text{and}\quad\Phi_2(t)=t^p\log(1+t)
	$$
	with $p>1$. Then $\Phi_1,\Phi_2\in\Delta_2\cap\nabla_2$ and $\Phi_1\prec\Phi_2$. However, if $\Phi\coloneq\Phi_2\circ\Phi_1^{-1}\in\nabla_2$, there is some $C=\nabla_2(\Phi)>1$ such that $\Phi(Ct)\geq 2C\Phi(t)$ for all $t\geq 0$. Thus,
	$$
	2C\leq\frac{\Phi(Ct)}{\Phi(t)}=C\cdot\frac{\log(1+(Ct)^{\frac{1}{p}})}{\log(1+t^{\frac{1}{p}})}\quad\text{for }t>0.
	$$
	Note that the fraction in the last term tends to $1$ as $t\to\infty$. Hence, there exists $t_0>0$ such that
	$$
	\frac{\log(1+(Ct)^{\frac{1}{p}})}{\log(1+t^{\frac{1}{p}})}<\frac{3}{2}
	$$
	whenever $t>t_0$, which yields a contradiction.
\end{remark}

\subsection{Parabolic mollification}\label{subsec: parabolic mollification}

For $z=(x,t)\in\mr^{n+1}$, define
$$
\kappa(z)=\kappa(x,t) \coloneq \tilde{\kappa}_n(|x|)\tilde{\kappa}_1(t),
$$
where $\tilde{\kappa}_m\in C^\infty (\mr)$ is given by
$$
\tilde{\kappa}_m(s) \coloneq \begin{cases}
	c_m\operatorname{exp}\left(\frac{1}{s^2-1}\right)\quad & \text{if }|s|<1,     \\
	0\quad                                                 & \text{if }|s|\geq 1,
\end{cases}
$$
with $c_m>0$ determined by $\int_{\mr^m}\tilde{\kappa}_m(|x|)\, dx =1$. Denote
$$
\kappa_h(x,t) \coloneq \frac{1}{h^{n+2}}\kappa\left( \frac{x}{h},\frac{t}{h^2} \right)\quad\text{for }h>0.
$$
We notice that $\kappa_h\in C^\infty(\mr^{n+1})$, $\kappa_h\geq 0$, $\|\kappa_h\|_{L^1(\mr^{n+1})}=1$ and $\operatorname{supp}(\kappa_h)\subset Q_h(0)$, where $Q_r(z)=B_r(x)\times I_{r^2}(t)$ for $r>0$ and $z=(x,t)\in\mr^{n+1}$. Moreover, $0\leq \kappa_h\leq c(n)h^{-n-2}$, $0\leq |D\kappa_h|\leq c(n)h^{-n-3}$ and $0\leq |\partial_t\kappa_h|\leq c(n)h^{-n-4}$.

We further define for $f\in L^1(U)$, where $U$ is a bounded open set in $\mr^{n+1}$,
$$
[f]^h (z) \coloneq (f * \kappa_h)(z)=\iints{\mr^{n+1}} f(z-\sigma)\kappa_h(\sigma)\,d\sigma=\iints{Q_h(0)} f(z-\sigma)\kappa_h(\sigma)\, d\sigma,
$$
where $Q_h(z)\subset U$.

\begin{proposition}[\cite{BYJ2026}, Proposition 2.1]    \label{prop: parabolic convolution}
	Let $U'\subset\subset U$ and $0<h<\frac{1}{2}\dist(\partial U,U')$. Then the following statements hold:
	\begin{enumerate}[label=\upshape(\roman*)]
		\item $[f]^h\in C^\infty(U')$ with $D^\alpha [f]^h=f*D^\alpha\kappa_h$;
		\item $[f]^h\to f$ a.e. in $U'$;
		\item if $f\in L^1(I;W^{k,1}(B))$ with $U'=B\times I=Q$, then $[D^\alpha f]^h=D^\alpha[f]^h$ on $Q$ for every multi-index $|\alpha|\leq k$.
	\end{enumerate}
\end{proposition}

Define $\Psi:\Omega_T\times\mr^n\to[0,\infty)$ by
$$
\Psi(z,\xi) \coloneq G(|\xi|)+a(z)H(|\xi|)
$$
for $z\in\Omega_T$ and $\xi\in\mr^n$. With a slight abuse of notation, we shall denote $\Psi(z,\xi)$ also when $\xi\in[0,\infty)$.

Note that $q \coloneq \log_2\Delta_2(G)>1$. By Lemma \ref{lem: properties for Young functions satisfying Delta_2 and nabla_2 conditions}\ref{lem: Delta_2 means aDec}, we have
$$
\frac{G(s)}{s^q}\leq\Delta_2(G)\frac{G(t)}{t^q}
$$
for $t\geq s>0$. This implies the following proposition.

\begin{proposition}[\cite{Hasto2025}, Proposition 4.1]    \label{prop: L^G convergence for parabolic mollifications}
	Let $U\subset\mr^{n+1}$ be a bounded open set and let $\kappa_h\in C_c^\infty(\mr^{n+1})$ and $[f]^h$ be as above. Then, for $f\in L_{\operatorname{loc}}^G(U)$,
	$$
	[f]^h\to f\quad\text{in }L_{\operatorname{loc}}^G(U)
	$$
	as $h\to 0^+$.
\end{proposition}

Indeed, Proposition \ref{prop: L^G convergence for parabolic mollifications} holds not only for $G$ but also for the functions $\varphi_p:[0,\infty)\to[0,\infty),\, p\geq 1$, defined by $\varphi_p(s)=s^p$.

\section{Proof of absence of the Lavrentiev phenomenon}\label{sec: proof of the main theorem}

\begin{proof}[Proof of Theorem \ref{thm: main theorem}]
	Let $R>0$ be the radius for the cylinder $Q$ and let $h_0>0$ be such that
	$$
	Q\equiv Q_R\subset\subset Q_{R+h_0}\subset\subset\Omega_T.
	$$
	From now on, we always assume that $h\in (0,h_0)$ so that $[w]^h$ is well-defined in $Q$. Define
	$$
	[w]^h(z) \coloneq (w*\kappa_h)(z),\quad a_h(z) \coloneq \inf_{Q_h(z)}a(\cdot),\quad \Psi_h(z,\xi) \coloneq G(|\xi|)+a_h(z)H(|\xi|)
	$$
	for $z\in Q$ and $\xi\in\mr^n$. Since $w\in L^1(0,T;W^{1,1}(\Omega))\subset L_{\loc}^1(\Omega_T)$ and $Dw\in L_{\loc}^G(\Omega_T)$ with $\mc{P}(w,\Omega_T)<\infty$, it follows from Proposition \ref{prop: L^G convergence for parabolic mollifications} that
	$$
	[w]^h\to w \quad\text{in }L^1(Q)\quad\text{and}\quad D[w]^h\to Dw\quad\text{in } L^G(Q).
	$$

	Observe that
	\begin{equation}    \label{eq: estimate Psi by H+Psi_h}
		\Psi(z,\xi)\leq|a(z)-a_h(z)|H(|\xi|)+\Psi_h(z,\xi)\leq c\omega(h)H(|\xi|)+\Psi_h(z,\xi)
	\end{equation}
	for every $z\in Q$ and $\xi\in\mr^n$; in particular, it is true for $\xi=D[w]^h(z)$. Jensen's inequality implies that
	\begin{equation}    \label{eq: estimate G by power of h}
		\begin{aligned}
			G(|D[w]^h(z)|) & \leq\iints{Q_h(0)}G(|Dw(z-\sigma)|)\kappa_h(\sigma)\,d\sigma \\
			& \leq c(n)h^{-n-2}\iints{\Omega_T}G(|Dw|)\,d\sigma\leq ch^{-n-2}
		\end{aligned}
	\end{equation}
	for $z\in Q$, where $c>0$ depends only on $n$ and $\|Dw\|_{L^G(\Omega_T)}$. Since the mapping
	$$
	t\mapsto\left( \frac{H}{G} \right)(t)
	$$
	is non-decreasing by Lemma \ref{lem: fraction of two Young functions}, it follows from Lemma \ref{lem: properties for Phi_1,Phi_2 satisfy Delta_2, nabla_2, Phi_1 < Phi_2} that
	\begin{equation}    \label{eq: estimate H by Psi_h}
		\begin{aligned}
		H(|D[w]^h(z)|) & =\left( \frac{H}{G} \right)(|D[w]^h(z)|)G(|D[w]^h(z)|) \\
		& \leq\left( \frac{H}{G} \right)(G^{-1}(ch^{-n-2}))G(|D[w]^h(z)|) \\
		& =\frac{(H\circ G^{-1})(ch^{-n-2})}{ch^{-n-2}}G(|D[w]^h(z)|) \\
		& \leq c\frac{(H\circ G^{-1})(h^{-n-2})}{h^{-n-2}}\Psi_h(z,D[w]^h(z)).
	\end{aligned}
	\end{equation}
	
	Therefore, we obtain from \eqref{eq: estimate Psi by H+Psi_h} that
	$$
	\begin{aligned}
		\Psi(z,D[w]^h(z)) & \leq c\left( \omega(h)\frac{(H\circ G^{-1})(h^{-n-2})}{h^{-n-2}}+1 \right)\Psi_h(z,D[w]^h(z)) \\
		& \leq c\Psi_h(z,D[w]^h(z))
	\end{aligned}
	$$
	with the assumption \eqref{eq: general gap condition}. Using Jensen's inequality once more, we have
	$$
	\begin{aligned}
		\Psi_h(z,D[w]^h(z)) & \leq\iints{Q_h(z)}\left[ G(|Dw(\sigma)|)+a_h(z)H(|Dw(\sigma)|) \right]\kappa_h(z-\sigma)\,d\sigma \\
		& \leq \iints{Q_h(z)}\Psi(\sigma,Dw(\sigma))\kappa_h(z-\sigma)\,d\sigma                                 \\
		& =[\Psi(\cdot,Dw(\cdot))*\kappa_h](z)=[\Psi(\cdot,Dw(\cdot))]^h(z).
	\end{aligned}
	$$
	Thus,
	$$
	\Psi(z,D[w]^h(z))\leq c[\Psi(\cdot,Dw(\cdot))]^h(z)
	$$
	for any $z\in Q$. Since $\mc{P}(w,\Omega_T)<\infty$,
	$$
	\Psi(\cdot,Dw(\cdot))\in L^1(\Omega_T)\subset L_{\loc}^1(\Omega_T).
	$$
	Applying Proposition \ref{prop: L^G convergence for parabolic mollifications}, we obtain that
	$$
	[\Psi(\cdot,Dw(\cdot))]^h\to\Psi(\cdot,Dw(\cdot))\quad\text{in }L^1(Q).
	$$
	Since $D[w]^h\to Dw$ a.e. in $Q$ with Proposition \ref{prop: parabolic convolution},
	$$
	\Psi(z,D[w]^h(z))\to\Psi(z,Dw(z))\quad\text{a.e. }z\in Q.
	$$
	A generalization of the dominated convergence theorem allows us to obtain a sequence $w_m \coloneq [w]^{h_m}\in C^\infty(Q)$ with $h_m\searrow 0$ that completes our proof.

	\medskip For the cases \ref{thm: main theorem bounded case} and \ref{thm: main theorem L^s case}, it suffices to perform calculations analogous to \eqref{eq: estimate G by power of h} and \eqref{eq: estimate H by Psi_h}. If $w\in L^\infty(\Omega_T)$ and \eqref{eq: bounded gap condition} holds, then we have
    $$
    |D[w]^h(z)|\leq\iints{Q_h(z)}|w(\sigma)||D\kappa_h(z-\sigma)|\,d\sigma\leq ch^{-n-3}\iints{Q_h(z)}|w(\sigma)|\,d\sigma\leq ch^{-1}
    $$
	and
	$$
	\begin{aligned}
		H(|D[w]^h(z)|)
		& = \left( \frac{H}{G} \right)(|D[w]^h(z)|)G(|D[w]^h(z)|) \\
		& \leq \left( \frac{H}{G} \right)(ch^{-1})G(|D[w]^h(z)|) \\
		& \leq c\frac{H(h^{-1})}{G(h^{-1})}\Psi_h(z,D[w]^h(z)).
	\end{aligned}
	$$
	Using \eqref{eq: estimate Psi by H+Psi_h}, we get
	$$
	\begin{aligned}
		\Psi(z,D[w]^h(z)) & \leq c\left( \omega(h)\frac{H(h^{-1})}{G(h^{-1})}+1 \right)\Psi_h(z,D[w]^h(z)) \\
		& \leq c\Psi_h(z,D[w]^h(z)).
	\end{aligned}
	$$
	The remaining part and the choice of the sequence are exactly the same as in the previous argument.

    Similarly, if $w\in L^\infty(0,T;L^s(\Omega))$ and \eqref{eq: L^infty(L^s) gap condition} holds, then it follows from H\"{o}lder's inequality that
	$$
	\begin{aligned}
        |D[w]^h(z)| & \leq\iints{Q_h(z)}|w(\sigma)||D\kappa_h(z-\sigma)|\,d\sigma \\
        & \leq ch^{-n-3}\iints{Q_h(z)}|w(\sigma)|\,d\sigma \\
        & \leq ch^{-n-3}\int_{I_{h^2}(t)}\left( \int_{B_h(x)}|w(y,\tau)|^s\,dy \right)^{\frac{1}{s}}\left( \int_{B_h(x)} 1 \,dy \right)^{\frac{s-1}{s}}d\tau \\
        & \leq ch^{-n-3}\int_{I_{h^2}(t)}\left( \int_{B_h(x)} 1 \,dy \right)^{\frac{s-1}{s}}d\tau \\
        & \leq ch^{-(n+s)/s}
    \end{aligned}
	$$
	and
	$$
	\begin{aligned}
		H(|D[w]^h(z)|) & = \left( \frac{H}{G} \right)(|D[w]^h(z)|)G(|D[w]^h(z)|) \\
		& \leq \left( \frac{H}{G} \right)(ch^{-(n+s)/s})G(|D[w]^h(z)|) \\
		&  \leq c\frac{H(h^{-(n+s)/s})}{G(h^{-(n+s)/s})}\Psi_h(z,D[w]^h(z)).
	\end{aligned}
	$$
	Therefore,
	$$
	\begin{aligned}
		\Psi(z,D[w]^h(z)) & \leq c\left( \omega(h)\frac{H(h^{-(n+s)/s})}{G(h^{-(n+s)/s})}+1 \right)\Psi_h(z,D[w]^h(z)) \\
		& \leq c\Psi_h(z,D[w]^h(z))
	\end{aligned}
	$$
	by \eqref{eq: estimate Psi by H+Psi_h} again. The remaining part is identical to the previous argument.
\end{proof}

\bibliographystyle{abbrv}
\bibliography{ref}{}

@Article{Wontae2023,
  author   = {Kim, Wontae and Kinnunen, Juha and Moring, Kristian},
  journal  = {Arch. Ration. Mech. Anal.},
  title    = {Gradient higher integrability for degenerate parabolic double-phase systems},
  year     = {2023},
  issn     = {0003-9527,1432-0673},
  number   = {5},
  pages    = {Paper No. 79, 46},
  volume   = {247},
  doi      = {10.1007/s00205-023-01918-0},
  fjournal = {Archive for Rational Mechanics and Analysis},
  mrclass  = {35J05},
  mrnumber = {4627284},
  url      = {https://doi.org/10.1007/s00205-023-01918-0},
}

@Article{Ok2017,
  author     = {Ok, Jihoon},
  journal    = {Calc. Var. Partial Differential Equations},
  title      = {Regularity of {$\omega$}-minimizers for a class of functionals with non-standard growth},
  year       = {2017},
  issn       = {0944-2669},
  number     = {2},
  pages      = {Paper No. 48, 31},
  volume     = {56},
  doi        = {10.1007/s00526-017-1137-5},
  fjournal   = {Calculus of Variations and Partial Differential Equations},
  mrclass    = {49N60 (35B65 35J20)},
  mrnumber   = {3626319},
  mrreviewer = {Xiaodong Yan},
  url        = {https://doi.org/10.1007/s00526-017-1137-5},
}

@Article{Colombo2015,
  author     = {Colombo, Maria and Mingione, Giuseppe},
  journal    = {Arch. Ration. Mech. Anal.},
  title      = {Regularity for double phase variational problems},
  year       = {2015},
  issn       = {0003-9527},
  number     = {2},
  pages      = {443--496},
  volume     = {215},
  doi        = {10.1007/s00205-014-0785-2},
  fjournal   = {Archive for Rational Mechanics and Analysis},
  mrclass    = {49N60 (35B27 35B65)},
  mrnumber   = {3294408},
  mrreviewer = {Eugen Viszus},
  url        = {https://doi.org/10.1007/s00205-014-0785-2},
}

@Article{Baroni2018,
  author     = {Baroni, Paolo and Colombo, Maria and Mingione, Giuseppe},
  journal    = {Calc. Var. Partial Differential Equations},
  title      = {Regularity for general functionals with double phase},
  year       = {2018},
  issn       = {0944-2669},
  number     = {2},
  pages      = {Paper No. 62, 48},
  volume     = {57},
  doi        = {10.1007/s00526-018-1332-z},
  fjournal   = {Calculus of Variations and Partial Differential Equations},
  mrclass    = {49N60},
  mrnumber   = {3775180},
  mrreviewer = {Elvira Mascolo},
  url        = {https://doi.org/10.1007/s00526-018-1332-z},
}

@Article{Baroni2015,
  author     = {Baroni, Paolo and Colombo, Maria and Mingione, Giuseppe},
  journal    = {Nonlinear Anal.},
  title      = {Harnack inequalities for double phase functionals},
  year       = {2015},
  issn       = {0362-546X},
  pages      = {206--222},
  volume     = {121},
  doi        = {10.1016/j.na.2014.11.001},
  fjournal   = {Nonlinear Analysis. Theory, Methods \& Applications. An International Multidisciplinary Journal},
  mrclass    = {49N60 (35J20)},
  mrnumber   = {3348922},
  mrreviewer = {Niko M. Marola},
  url        = {https://doi.org/10.1016/j.na.2014.11.001},
}

@Article{Mingione2021,
  author     = {Mingione, Giuseppe and R\v{a}dulescu, Vicen\c{t}iu},
  journal    = {J. Math. Anal. Appl.},
  title      = {Recent developments in problems with nonstandard growth and nonuniform ellipticity},
  year       = {2021},
  issn       = {0022-247X},
  number     = {1},
  pages      = {Paper No. 125197, 41},
  volume     = {501},
  doi        = {10.1016/j.jmaa.2021.125197},
  fjournal   = {Journal of Mathematical Analysis and Applications},
  mrclass    = {49-02 (35J50 49J10 49N60)},
  mrnumber   = {4258810},
  mrreviewer = {Carlo Mariconda},
  url        = {https://doi.org/10.1016/j.jmaa.2021.125197},
}

@Article{Zhikov1995,
  author     = {Zhikov, Vasili\u{\i} V.},
  journal    = {Russian J. Math. Phys.},
  title      = {On {L}avrentiev's phenomenon},
  year       = {1995},
  issn       = {1061-9208},
  number     = {2},
  pages      = {249--269},
  volume     = {3},
  fjournal   = {Russian Journal of Mathematical Physics},
  mrclass    = {49J45 (49J10)},
  mrnumber   = {1350506},
  mrreviewer = {Philip D. Loewen},
}

@Article{Zhikov1986,
  author     = {Zhikov, V. V.},
  journal    = {Izv. Akad. Nauk SSSR Ser. Mat.},
  title      = {Averaging of functionals of the calculus of variations and elasticity theory},
  year       = {1986},
  issn       = {0373-2436},
  number     = {4},
  pages      = {675--710, 877},
  volume     = {50},
  fjournal   = {Izvestiya Akademii Nauk SSSR. Seriya Matematicheskaya},
  mrclass    = {49H05 (73C60)},
  mrnumber   = {864171},
  mrreviewer = {Vadim Komkov},
}

@Article{Zhikov1997,
  author     = {Zhikov, Vasili\u{\i} V.},
  journal    = {Russian J. Math. Phys.},
  title      = {On some variational problems},
  year       = {1997},
  issn       = {1061-9208},
  number     = {1},
  pages      = {105--116 (1998)},
  volume     = {5},
  fjournal   = {Russian Journal of Mathematical Physics},
  mrclass    = {49J10 (49J45 49N60)},
  mrnumber   = {1486765},
  mrreviewer = {Francesco Ferro},
}

@Article{Zhikov1993,
  author     = {Zhikov, Vasili\u{\i} V.},
  journal    = {C. R. Acad. Sci. Paris S\'{e}r. I Math.},
  title      = {Lavrentiev phenomenon and homogenization for some variational problems},
  year       = {1993},
  issn       = {0764-4442},
  number     = {5},
  pages      = {435--439},
  volume     = {316},
  fjournal   = {Comptes Rendus de l'Acad\'{e}mie des Sciences. S\'{e}rie I. Math\'{e}matique},
  mrclass    = {49J45 (35B27 73B27)},
  mrnumber   = {1209262},
  mrreviewer = {J. Saint Jean Paulin},
}

@Article{Haestoe2022,
  author   = {H\"{a}st\"{o}, Peter and Ok, Jihoon},
  journal  = {Arch. Ration. Mech. Anal.},
  title    = {Regularity theory for non-autonomous partial differential equations without {U}hlenbeck structure},
  year     = {2022},
  issn     = {0003-9527},
  number   = {3},
  pages    = {1401--1436},
  volume   = {245},
  doi      = {10.1007/s00205-022-01807-y},
  fjournal = {Archive for Rational Mechanics and Analysis},
  mrclass  = {49N60 (35D30 49J10)},
  mrnumber = {4467321},
  url      = {https://doi.org/10.1007/s00205-022-01807-y},
}

@Article{Ok2020,
  author     = {Ok, Jihoon},
  journal    = {Nonlinear Anal.},
  title      = {Regularity for double phase problems under additional integrability assumptions},
  year       = {2020},
  issn       = {0362-546X},
  pages      = {111408, 13},
  volume     = {194},
  doi        = {10.1016/j.na.2018.12.019},
  fjournal   = {Nonlinear Analysis. Theory, Methods \& Applications. An International Multidisciplinary Journal},
  mrclass    = {49N60 (35D30 35J62 35J70)},
  mrnumber   = {4074606},
  mrreviewer = {Elvira Mascolo},
  url        = {https://doi.org/10.1016/j.na.2018.12.019},
}

@Article{Haestoe2022a,
  author     = {H\"{a}st\"{o}, Peter and Ok, Jihoon},
  journal    = {J. Eur. Math. Soc. (JEMS)},
  title      = {Maximal regularity for local minimizers of non-autonomous functionals},
  year       = {2022},
  issn       = {1435-9855},
  number     = {4},
  pages      = {1285--1334},
  volume     = {24},
  doi        = {10.4171/JEMS/1118},
  fjournal   = {Journal of the European Mathematical Society (JEMS)},
  mrclass    = {49N60 (35A15 35J62 46E30 49J10)},
  mrnumber   = {4397041},
  mrreviewer = {Antonia Passarelli di Napoli},
  url        = {https://doi.org/10.4171/JEMS/1118},
}

@Article{Byun2021,
  author   = {Byun, Sun-Sig and Lee, Ho-Sik},
  journal  = {Q. J. Math.},
  title    = {Gradient estimates of {$\omega$}-minimizers to double phase variational problems with variable exponents},
  year     = {2021},
  issn     = {0033-5606},
  number   = {4},
  pages    = {1191--1221},
  volume   = {72},
  doi      = {10.1093/qmath/haaa067},
  fjournal = {The Quarterly Journal of Mathematics},
  mrclass  = {35A15},
  mrnumber = {4350146},
  url      = {https://doi.org/10.1093/qmath/haaa067},
}

@Article{Byun2021a,
  author   = {Byun, Sun-Sig and Lee, Ho-Sik},
  journal  = {J. Math. Anal. Appl.},
  title    = {Calder\'{o}n-{Z}ygmund estimates for elliptic double phase problems with variable exponents},
  year     = {2021},
  issn     = {0022-247X},
  number   = {1},
  pages    = {Paper No. 124015, 31},
  volume   = {501},
  doi      = {10.1016/j.jmaa.2020.124015},
  fjournal = {Journal of Mathematical Analysis and Applications},
  mrclass  = {35J30},
  mrnumber = {4258791},
  url      = {https://doi.org/10.1016/j.jmaa.2020.124015},
}

@Article{Marcellini2020,
  author   = {Marcellini, Paolo},
  journal  = {Nonlinear Anal.},
  title    = {A variational approach to parabolic equations under general and {$p,q$}-growth conditions},
  year     = {2020},
  issn     = {0362-546X},
  pages    = {111456, 17},
  volume   = {194},
  doi      = {10.1016/j.na.2019.02.010},
  fjournal = {Nonlinear Analysis. Theory, Methods \& Applications. An International Multidisciplinary Journal},
  mrclass  = {35K59 (35B65 35K20 49N60)},
  mrnumber = {4074614},
  url      = {https://doi.org/10.1016/j.na.2019.02.010},
}

@Article{Colombo2015a,
  author     = {Colombo, Maria and Mingione, Giuseppe},
  journal    = {Arch. Ration. Mech. Anal.},
  title      = {Bounded minimisers of double phase variational integrals},
  year       = {2015},
  issn       = {0003-9527},
  number     = {1},
  pages      = {219--273},
  volume     = {218},
  doi        = {10.1007/s00205-015-0859-9},
  fjournal   = {Archive for Rational Mechanics and Analysis},
  mrclass    = {49J10 (49N60)},
  mrnumber   = {3360738},
  mrreviewer = {Helmut Kaul},
  url        = {https://doi.org/10.1007/s00205-015-0859-9},
}

@Article{Esposito2004,
  author     = {Esposito, Luca and Leonetti, Francesco and Mingione, Giuseppe},
  journal    = {J. Differential Equations},
  title      = {Sharp regularity for functionals with {$(p,q)$} growth},
  year       = {2004},
  issn       = {0022-0396},
  number     = {1},
  pages      = {5--55},
  volume     = {204},
  doi        = {10.1016/j.jde.2003.11.007},
  fjournal   = {Journal of Differential Equations},
  mrclass    = {49J10 (49N60)},
  mrnumber   = {2076158},
  mrreviewer = {Delfim F. M. Torres},
  url        = {https://doi.org/10.1016/j.jde.2003.11.007},
}

@Article{Fonseca2004,
  author     = {Fonseca, Irene and Mal\'{y}, Jan and Mingione, Giuseppe},
  journal    = {Arch. Ration. Mech. Anal.},
  title      = {Scalar minimizers with fractal singular sets},
  year       = {2004},
  issn       = {0003-9527},
  number     = {2},
  pages      = {295--307},
  volume     = {172},
  doi        = {10.1007/s00205-003-0301-6},
  fjournal   = {Archive for Rational Mechanics and Analysis},
  mrclass    = {49J10},
  mrnumber   = {2058167},
  mrreviewer = {Elvira Mascolo},
  url        = {https://doi.org/10.1007/s00205-003-0301-6},
}

@Article{Baasandorj2020,
  author     = {Baasandorj, Sumiya and Byun, Sun-Sig and Oh, Jehan},
  journal    = {J. Funct. Anal.},
  title      = {Calder\'{o}n-{Z}ygmund estimates for generalized double phase problems},
  year       = {2020},
  issn       = {0022-1236},
  number     = {7},
  pages      = {108670, 57},
  volume     = {279},
  doi        = {10.1016/j.jfa.2020.108670},
  fjournal   = {Journal of Functional Analysis},
  mrclass    = {35J70 (35B65 42B25 46E35)},
  mrnumber   = {4107816},
  mrreviewer = {Juha K. Kinnunen},
  url        = {https://doi.org/10.1016/j.jfa.2020.108670},
}

@Article{Byun2020,
  author     = {Byun, Sun-Sig and Oh, Jehan},
  journal    = {Anal. PDE},
  title      = {Regularity results for generalized double phase functionals},
  year       = {2020},
  issn       = {2157-5045},
  number     = {5},
  pages      = {1269--1300},
  volume     = {13},
  doi        = {10.2140/apde.2020.13.1269},
  fjournal   = {Analysis \& PDE},
  mrclass    = {49N60 (35B65 35J20 49J10)},
  mrnumber   = {4149062},
  mrreviewer = {Elvira Mascolo},
  url        = {https://doi.org/10.2140/apde.2020.13.1269},
}

@Article{Byun2017,
  author   = {Byun, Sun-Sig and Oh, Jehan},
  journal  = {Calc. Var. Partial Differential Equations},
  title    = {Global gradient estimates for non-uniformly elliptic equations},
  year     = {2017},
  issn     = {0944-2669},
  number   = {2},
  pages    = {Paper No. 46, 36},
  volume   = {56},
  doi      = {10.1007/s00526-017-1148-2},
  fjournal = {Calculus of Variations and Partial Differential Equations},
  mrclass  = {35J62 (35B65 35J70)},
  mrnumber = {3624942},
  url      = {https://doi.org/10.1007/s00526-017-1148-2},
}

@Article{Marcellini1989,
  author  = {Marcellini, Paolo},
  journal = {Archive for Rational Mechanics and Analysis},
  title   = {Regularity of Minimizers of Integrals of the Calculus of Variations with Non–Standard Growth Conditions},
  year    = {1989},
  month   = {01},
  pages   = {267-284},
  volume  = {105},
  doi     = {10.1007/BF00251503},
}

@Article{Colombo2016,
  author     = {Colombo, Maria and Mingione, Giuseppe},
  journal    = {J. Funct. Anal.},
  title      = {Calder\'{o}n-{Z}ygmund estimates and non-uniformly elliptic operators},
  year       = {2016},
  issn       = {0022-1236,1096-0783},
  number     = {4},
  pages      = {1416--1478},
  volume     = {270},
  doi        = {10.1016/j.jfa.2015.06.022},
  fjournal   = {Journal of Functional Analysis},
  mrclass    = {35J62 (35B65 35J70)},
  mrnumber   = {3447716},
  mrreviewer = {Francesco\ Della Pietra},
  url        = {https://doi.org/10.1016/j.jfa.2015.06.022},
}

@Article{Chlebicka2019a,
  author     = {Chlebicka, Iwona and Gwiazda, Piotr and Zatorska-Goldstein, Anna},
  journal    = {Ann. Inst. H. Poincar\'{e} C Anal. Non Lin\'{e}aire},
  title      = {Parabolic equation in time and space dependent anisotropic {M}usielak-{O}rlicz spaces in absence of {L}avrentiev's phenomenon},
  year       = {2019},
  issn       = {0294-1449,1873-1430},
  number     = {5},
  pages      = {1431--1465},
  volume     = {36},
  doi        = {10.1016/j.anihpc.2019.01.003},
  fjournal   = {Annales de l'Institut Henri Poincar\'{e} C. Analyse Non Lin\'{e}aire},
  mrclass    = {35K59 (35A01 35K20)},
  mrnumber   = {3985549},
  mrreviewer = {Daniele\ Andreucci},
  url        = {https://doi.org/10.1016/j.anihpc.2019.01.003},
}

@Article{Wontae2024,
  author   = {Kim, Wontae and S\"arki\"o, Lauri},
  journal  = {NoDEA Nonlinear Differential Equations Appl.},
  title    = {Gradient higher integrability for singular parabolic double-phase systems},
  year     = {2024},
  issn     = {1021-9722,1420-9004},
  number   = {3},
  pages    = {Paper No. 40, 38},
  volume   = {31},
  doi      = {10.1007/s00030-024-00928-5},
  fjournal = {NoDEA. Nonlinear Differential Equations and Applications},
  mrclass  = {35D30 (35K55 35K65)},
  mrnumber = {4718687},
  url      = {https://doi.org/10.1007/s00030-024-00928-5},
}

@Article{Kim2025b,
  author  = {Kim, Bogi and Oh, Jehan},
  journal = {arXiv preprint arXiv:2511.13454},
  title   = {Bounded solutions and interpolative gap bounds for degenerate parabolic double phase problems},
  year    = {2025},
}

@Article{Kim2026,
  author   = {Kim, Bogi and Kim, Youngchae and Oh, Jehan},
  journal  = {NoDEA Nonlinear Differential Equations Appl.},
  title    = {Gradient estimates for double phase problems with two modulating coefficients},
  year     = {2026},
  issn     = {1021-9722,1420-9004},
  number   = {1},
  pages    = {Paper No. 12},
  volume   = {33},
  doi      = {10.1007/s00030-025-01162-3},
  fjournal = {NoDEA. Nonlinear Differential Equations and Applications},
  mrclass  = {35B65 (35A15 35D30 35J70)},
  mrnumber = {4984383},
  url      = {https://doi.org/10.1007/s00030-025-01162-3},
}

@Article{Kim2025a,
  author   = {Kim, Bogi and Oh, Jehan},
  journal  = {Adv. Nonlinear Anal.},
  title    = {Regularity for double-phase functionals with nearly linear growth and two modulating coefficients},
  year     = {2025},
  issn     = {2191-9496,2191-950X},
  number   = {1},
  pages    = {Paper No. 20250090, 19},
  volume   = {14},
  doi      = {10.1515/anona-2025-0090},
  fjournal = {Advances in Nonlinear Analysis},
  mrclass  = {35B65 (35A15 35J70 49J10 49N60)},
  mrnumber = {4926894},
  url      = {https://doi.org/10.1515/anona-2025-0090},
}

@Article{Kim2024a,
  author     = {Kim, Bogi and Oh, Jehan},
  journal    = {J. Geom. Anal.},
  title      = {Regularity for double phase functionals with two modulating coefficients},
  year       = {2024},
  issn       = {1050-6926,1559-002X},
  number     = {5},
  pages      = {Paper No. 134, 51},
  volume     = {34},
  doi        = {10.1007/s12220-024-01584-y},
  fjournal   = {Journal of Geometric Analysis},
  mrclass    = {35B65 (35J70)},
  mrnumber   = {4719982},
  mrreviewer = {Marcos\ Montenegro},
  url        = {https://doi.org/10.1007/s12220-024-01584-y},
}

@Article{Hasto2025,
  author  = {H{\"a}st{\"o}, Peter and Ok, Jihoon},
  journal = {arXiv preprint arXiv:2511.19758},
  title   = {Higher integrability for parabolic PDEs with generalized Orlicz growth},
  year    = {2025},
}

@Article{BYJ2026,
  author  = {Kim, Bogi and Kim, Youngchae and Oh, Jehan},
  journal = {arXiv e-prints},
  title   = {Absence of the Lavrentiev phenomenon for degenerate parabolic double phase problems},
  year    = {2026},
  pages   = {arXiv--2603},
}

@Article{Maligranda2004,
  author     = {Maligranda, L.},
  journal    = {Acta Math. Hungar.},
  title      = {Calder\'on-{L}ozanovski\u i\ construction for mixed norm spaces},
  year       = {2004},
  issn       = {0236-5294,1588-2632},
  number     = {4},
  pages      = {279--302},
  volume     = {103},
  doi        = {10.1023/B:AMHU.0000028829.15720.02},
  fjournal   = {Acta Mathematica Hungarica},
  mrclass    = {46E30 (46B42 46B70)},
  mrnumber   = {2062631},
  mrreviewer = {Fernando\ Cobos},
  url        = {https://doi.org/10.1023/B:AMHU.0000028829.15720.02},
}

@Article{Chlebicka2019b,
  author     = {Chlebicka, Iwona and Gwiazda, Piotr and Zatorska-Goldstein, Anna},
  journal    = {J. Differential Equations},
  title      = {Renormalized solutions to parabolic equations in time and space dependent anisotropic {M}usielak-{O}rlicz spaces in absence of {L}avrentiev's phenomenon},
  year       = {2019},
  issn       = {0022-0396,1090-2732},
  number     = {2},
  pages      = {1129--1166},
  volume     = {267},
  doi        = {10.1016/j.jde.2019.02.005},
  fjournal   = {Journal of Differential Equations},
  mrclass    = {35K59 (35A01 35R05)},
  mrnumber   = {3957983},
  mrreviewer = {Piotr\ Rybka},
  url        = {https://doi.org/10.1016/j.jde.2019.02.005},
}

@Article{Bulicek2022,
  author   = {Bul\'{i}\v{c}ek, Miroslav and Gwiazda, Piotr and Skrzeczkowski, Jakub},
  journal  = {Arch. Ration. Mech. Anal.},
  title    = {On a range of exponents for absence of {L}avrentiev phenomenon for double phase functionals},
  year     = {2022},
  issn     = {0003-9527,1432-0673},
  number   = {1},
  pages    = {209--240},
  volume   = {246},
  doi      = {10.1007/s00205-022-01816-x},
  fjournal = {Archive for Rational Mechanics and Analysis},
  mrclass  = {49J10 (49J45)},
  mrnumber = {4487513},
  url      = {https://doi.org/10.1007/s00205-022-01816-x},
}

@Article{DeFilippis2019,
  author   = {De Filippis, C. and Mingione, G.},
  journal  = {Algebra i Analiz},
  title    = {A borderline case of {C}alder\'on-{Z}ygmund estimates for nonuniformly elliptic problems},
  year     = {2019},
  issn     = {0234-0852},
  number   = {3},
  pages    = {82--115},
  volume   = {31},
  doi      = {10.1090/spmj/1608},
  fjournal = {Rossi\u iskaya Akademiya Nauk. Algebra i Analiz},
  mrclass  = {35J70 (35B27 35B45 35J62 35J92 49J10 49N60)},
  mrnumber = {3985927},
  url      = {https://doi.org/10.1090/spmj/1608},
}

@Article{DeFilippis2023,
  author     = {De Filippis, Cristiana and Mingione, Giuseppe},
  journal    = {Arch. Ration. Mech. Anal.},
  title      = {Regularity for double phase problems at nearly linear growth},
  year       = {2023},
  issn       = {0003-9527,1432-0673},
  number     = {5},
  pages      = {Paper No. 85, 50},
  volume     = {247},
  doi        = {10.1007/s00205-023-01907-3},
  fjournal   = {Archive for Rational Mechanics and Analysis},
  mrclass    = {49J10 (35Q49 49N60)},
  mrnumber   = {4630451},
  mrreviewer = {Antonio\ Leaci},
  url        = {https://doi.org/10.1007/s00205-023-01907-3},
}

@Book{Chlebicka2021,
  author    = {Chlebicka, Iwona and Gwiazda, Piotr and \'Swierczewska-Gwiazda, Agnieszka and Wr\'oblewska-Kami\'nska, Aneta},
  publisher = {Springer, Cham},
  title     = {Partial differential equations in anisotropic {M}usielak-{O}rlicz spaces},
  year      = {[2021] \copyright 2021},
  isbn      = {978-3-030-88855-8; 978-3-030-88856-5},
  series    = {Springer Monographs in Mathematics},
  doi       = {10.1007/978-3-030-88856-5},
  mrclass   = {35-02 (35A01 46E30 76D03)},
  mrnumber  = {4357588},
  pages     = {xiii+389},
  url       = {https://doi.org/10.1007/978-3-030-88856-5},
}

@Article{Chlebicka2025a,
  author  = {Chlebicka, Iwona and Garain, Prashanta and Kim, Wontae},
  journal = {arXiv preprint arXiv:2512.11294},
  title   = {Gradient higher integrability of bounded solutions to parabolic double-phase systems},
  year    = {2025},
}

@Article{Kim2026a,
  author  = {Kim, Bogi and Oh, Jehan},
  journal = {arXiv preprint arXiv:2601.01571},
  title   = {Interpolative Refinement of Gap Bound Conditions for Singular Parabolic Double Phase Problems},
  year    = {2026},
}

@Article{Marcellini1991,
  author     = {Marcellini, Paolo},
  journal    = {J. Differential Equations},
  title      = {Regularity and existence of solutions of elliptic equations with {$p,q$}-growth conditions},
  year       = {1991},
  issn       = {0022-0396,1090-2732},
  number     = {1},
  pages      = {1--30},
  volume     = {90},
  doi        = {10.1016/0022-0396(91)90158-6},
  fjournal   = {Journal of Differential Equations},
  mrclass    = {35J15 (35D10)},
  mrnumber   = {1094446},
  mrreviewer = {Philip\ W.\ Schaefer},
  url        = {https://doi.org/10.1016/0022-0396(91)90158-6},
}
\end{document}